\documentclass[11pt,reqno]{amsart}
\usepackage{amssymb,latexsym}
\usepackage{graphicx}
\usepackage{fancyhdr,amssymb}
\usepackage{url}        

\theoremstyle{plain} \numberwithin{equation}{section}
\newtheorem{thm}{Theorem}[section]
\newtheorem{theorem}[thm]{Theorem}
\newtheorem{lemma}[thm]{Lemma}

\begin{document}

\fancyhead{}
\renewcommand{\headrulewidth}{0pt}

\setcounter{page}{1}

\title[Repdigits in Euler functions of Pell and associated pell numbers]{Repdigits in Euler functions of Pell and associated pell numbers}
\author{M.K. Sahukar}
\address{Department of Mathematics\\
                National Institute of Technology\\
                Rourkela,orissa\\
                India}
\email{manasi.sahukar@gmail.com}

\author{G.K. Panda}
\address{Department of Mathematics\\
               National Institute of Technology\\
                Rourkela,orissa\\
                India}
\email{gkpanda\_nit@rediffmail.com}

\begin{abstract}
 A natural number $n$ is called a repdigit if all its digits are same. In this paper, we prove that Euler totient function of no Pell number is a repdigit with at least two digits. This study is also extended to certain subclass of associated Pell numbers. 
\end{abstract}

\maketitle
\maketitle 
\textbf{2010 Subject classification [A.M.S.]:} 11A25, 11D61, 11B39.

\section{introduction}
The Pell sequence $\{P_n\}_{n\geq0}$ and the associated Pell sequence $\{Q_n\}_{n\geq0}$ are defined by the binary recurrences 
\begin{align*}
P_{n+1}=2P_{n}+P_{n-1}~~and~~Q_{n+1}=2Q_{n}+Q_{n-1},
\end{align*} 
with the initial terms $P_0=0,P_1=1$ and $Q_0=1,Q_1=1$ respectively. If $\alpha=1+\sqrt{2}$ and $\beta=1-\sqrt{2}$, then $P_n=\frac{\alpha^n-\beta^n}{\alpha-\beta}$ and $Q_n=\frac{\alpha^n+\beta^n}{2}$ for all $n\geq0$. The Euler totient function $\phi(n)$ of a positive integer $n$ is the number of positive integers less than or equal to $n$ and relatively prime to $n$. If $n$ has the canonical decomposition $n=p_1^{a_1}\cdots p_r^{a_r}$, then it is well-known that
$$\phi(n) = p_1^{a_1-1}(p_1-1)\cdots p_r^{a_r-1}(p_r-1).$$ 
In \cite{Luca 2014}, it was shown that if the Euler function of the $n^{th}$ Pell
number $P_n$ or associated Pell number $Q_n$ is a power of $2$ then $n\leq 8$. In 2014, Damir and Faye\cite{Damir 2014} proved that  if $\{u_n\}_{n\geq0}$ is the Lucas sequence defined by $u_0=0$, $u_1=1$ and $u_{n+2}=ru_{n+1}+su_n$ for all $n \leq 0$ with $s\in \{1,-1\}$, then there are finitely many $n$ such that $\phi(|u_n|)$ is a power of 2.

Recall that a positive integer is called a repdigit if it has only one distinct digit in its decimal expansion. In particular, such numbers are of the form $d(10^m-1)/9$ for some $m\geq1$ and $1\leq d \leq9$. In \cite{Faye 2015}, it was shown that there is no repdigit Pell or Pell-Lucas number larger than 10.

The study of repdigits in Euler functions of specified number sequences has attracted several number theorists. In 2002, Luca \cite{Luca 2002} proved that there exist only finitely many positive integer solution $(m,n)$ satisfying the Diophantine equation $\phi(U_n)=V_m$ where $\{U_n\}_{n\geq0}$ and $\{V_m\}_{m\geq0}$ are two non-degenerate binary recurrence sequences. Taking $V_m=d\cdot\frac{10^m-1}{9}$ where $d\in\{1,2,\cdots, 9\}$, Luca\cite{Bravo 2016,Luca 2006}  investigated the presence of repdigits associated with the Euler functions of Fibonacci and  Lucas numbers. In this paper, we follow the method described in \cite{Bravo 2016,Luca 2006} to investigate the presence of repdigits with atleast two digits in the Euler functions of Pell  and associated Pell numbers.  

Through out this paper, we use $p$ with or without subscripts as a prime number, $\big(\frac{a}{b}\big)$ as Legendre symbol of $a$ and $b$ and $(a,b)$ as the greatest common divisor of $a$ and $b$.\\
\section{Preliminaries}
To achieve the objective of this paper, we need the following results and definitions. We shall keep on referring this section with or without further reference.

\begin{lemma}\label{2.1}\
If $m$ and $n$ be natural numbers, then
\begin{enumerate}
\item $P_{2n}=2P_nQ_n,$\\
\item $Q_n^2-2P_n^2=(-1)^n$,\\
\item $(P_n,Q_n)=1$,\\
\item $P_m|P_n$ if and only if $m|n$,\\
\item $Q_m|Q_n$ if and only if $m|n$ and $\frac{n}{m}$ is odd,\\
\item $v_2(P_n)=v_2(n)$ and $v_2(Q_n)=0$ where $v_2(n)$ is the exponent of $2$ in the canonical decomposition  of $n$,\\
\item $3|Q_n$ if and only if $n\equiv2(\bmod~4)$,\\
\item $5\nmid Q_n$ for any $n$,\\
\item $Q_{3\cdot 2^t}=Q_{2^t}(4Q_{2^t}^2-3)$.

\end{enumerate}
\end{lemma}
For the proof of this lemma the readers are advised to refer to \cite{koshy}.

\begin{lemma}\label{2.2}{\rm{(\cite{Petho 1992},Theorem 2,\cite{Cohn 1996}, Theorem 1)}}.
The only solutions of Diophantine equation $P_n=y^m$ in positive integers $n,y,m$ with $m\geq2$ are $(n,y,m)= (1,1,m)$, $(7,13,2)$.
\end{lemma}
\begin{lemma}\label{2.3}{\rm{(\cite{Daniel 1998},Theorem 1)}}. 
The solutions of the Diophantine equation $P_mP_n=x^2$ with $1\leq m<n$ are $(m,n)=(1,7)$ or $n=3m$, $3\nmid m$, $m$ is odd.
\end{lemma}
\begin{lemma}\label{2.4}{\rm{(\cite{Bravo 2015}, Theorem A)}}.
 If $n,y,m$ are positive integers with $m\geq2$, then the equation $Q_n=y^m$ has the only solution $(n,y)=(1,1)$.
\end{lemma}
\begin{lemma}\label{2.5}{\rm{(\cite{Daniel 1998}, Theorem 3)}}.
The solutions of the Diophantine equation $Q_mQ_n=x^2$ with $0\leq m<n$ are $n=3m$, $3\nmid m$, $m$ is odd.
\end{lemma}

\begin{lemma}\label{2.6}
 If $m$ and $n$ are positive integers and $p$ is an odd prime then the Diophantine equation $P_n=4p^m$ has only one integer solution $n=4, p=3, m=1$.
\end{lemma}
\begin{proof}
Suppose that $P_n=4p^m$ where $p$ is a prime and $m$ and $n$ are positive integers. Since $4|P_n$, $n=4k$ for some $k$. Hence $P_n=P_{4k}=2P_{2k}Q_{2k}=4p^m$. Since $(P_{2k}, Q_{2k})=1$ and $Q_n$ is odd for all $n\geq0$, it follows that $P_{2k}=2$ and $Q_{2k}=p^m$. 
\end{proof}

\begin{lemma}{\rm{(\cite{Bravo 2015}, Lemma 2.1)}}.\label{2.7}
Let $(u_n)_{n\geq0}$ be a binary recurrence sequence. If there exists a prime $p$ such that $p|u_n$ and $p\nmid \displaystyle\prod_{i=1}^{n-1}u_i$, then $p$ is called as primitive prime factor of $u_n$ and is always congruent to $\pm1$ modulo $n$.
\end{lemma}
\begin{lemma}\label{2.8}{\rm{(\cite{Bravo 2015}, Lemma 2.1)}}.
 A Primitive prime factor of $P_n$ exists if $n\geq 3$ and that of $Q_n$ exists if $n\geq 2$.
 \end{lemma}
 \begin{lemma}\label{2.9}{\rm{( \cite{Daniel 2000}, Theorem 4)}}
 There exist a prime factor $p$ of $P_n$ such that $p\equiv1(\bmod~4)$ if $n \neq 0,1,2,4,14$.
 \end{lemma}
 \begin{lemma}\label{2.10}{\rm{(\cite{pell}, Pell-Lucas numbers)}}.
 If the associated Pell number $Q_n$ is a prime then $n$ is either a prime or a power of $2$. 
 \end{lemma}
 \begin{lemma} \label{2.11}{\rm{(\cite{Luca 1997}, Lemma 3)}}.
 If $n\neq 1,2,6$ then $\phi(n)\geq 2\sqrt{\frac{n}{3}}$.
 \end{lemma}
 
\section{Repdigits in Euler functions of Pell numbers}

We start the section by computing the least residues and periods of the Pell sequence $\{P_n\}_{n\geq0}$ modulo 11, 20, 40 to use then subsequently

\begin{table}[h!]
  \centering
 \caption{Periods of $P_n$}
  \label{tab:table1}
  \begin{tabular}{|l|c|r|}
  \hline
    $k$ & $P_n(\bmod~k)$ & Period\\
    \hline
    11 & 0, 1, 2, 5, 1, 7, 4, 4, 1, 6, 2, 10, 0, 10, 9, 6, 10, 4, 7, 7, 10, 5, 9, 1 & 24\\
    \hline
    20 &  0, 1, 2, 5, 12, 9, 10, 9, 8, 5, 18, 1 & 12\\
    \hline
    40 & 0, 1, 2, 5, 12, 29, 30, 9, 8, 25, 18, 21, 20, 21, 22, 25, 32, 9, 10, 29, 28, 5, 38,1 & 24\\
    \hline
  \end{tabular}
\end{table}
The following theorem which proves the nonexistence of repdigits with atleast two digits in the Euler function of Pell numbers is one of the main results of this paper. 
\begin{theorem}\label{3.1}
The equation 
\begin{equation} \label{3.1}
\phi{(P_n)}=d\cdot\frac{10^m-1}{9}
\end{equation}
has no solution in the positive integers $n,m,d$ such that $m\geq2$ and $d\in\{1,2,\ldots,9\}$.
\end{theorem}

\begin{proof}
For $n\leq16$, it is easy to see that there is no Pell number $P_n$ such that $\phi(P_n)$ is a repdigit with atleast two digits. Assume to the contrary that for some $n>16$, $\phi(P_n)$ is a repdigit, that is
\begin{equation*}
\phi{(P_n)}=d\cdot\frac{10^m-1}{9}
\end{equation*}
for some $d\in \{1,2,\ldots,9\}$ and for some $n$. Let $v_2(n)$ be the exponent of 2 in the factorization of a positive integer $n$. Since $\frac{10^m-1}{9}$ is odd, it follows that
\begin{equation}\label{3.2}
v_2(\phi{(P_n)})=v_2(d)\leq 3.
\end{equation}
By virtue of Lemma \ref{2.8}, there always exists a prime factor $p_1$ of $P_n$ such that $p_1 \equiv 1(\bmod~4)$. Clearly, $p_1-1|\phi{(P_n)}$ and $v_2(d)\geq 2$ which implies that either there exists another odd prime factor $p_2$ of $P_n$ such that $p_2\equiv3(\bmod~4)$ or $p_1$ is the only odd prime factor of $P_n$.\\

\par
Firstly, assume that $P_n$ has two distinct prime factors $p_1$ and $p_2$ such that $p_1\equiv 1(\bmod~4)$, $p_2\equiv 3(\bmod~4)$. If $n$ is odd then reducing the relation (2) in Lemma \ref{2.1} modulo $p_2$, we get $Q_n^2 \equiv -1(\bmod ~p_2)$ which implies that $-1$ is a quadratic residue modulo $p_2$. But this is possible only when $p_2 \equiv 1(\bmod ~4)$ which is a contradiction to $p_2\equiv3(\bmod~4)$. If $n$ is even then $P_n$ is even. Let $P_n=2^a\cdot p^b_1\cdot p^c_2$. If $a >1$, then $2(p_1-1)(p_2-1)|\phi(P_n)$, which implies $v_2(\phi(P_n)) \geq 4$, which contradicts (\ref{3.2}). Hence, $a=1$ and consequently $P_n=2\cdot p_1^b\cdot p_2^c$. Let $n=2n_1$. If $n_1$ is even, then $4|n$ and hence $4|P_n$ leading to $a\geq 2$. Hence, $n_1$ must be odd. Since $P_{2n_1}=2P_{n_1}Q_{n_1}$ and $(P_{n_1},Q_{n_1})=1$ from Lemma \ref{2.1}, it follows that $P_{n_1}=p_1^b$ and $Q_{n_1}=p_2^c$. Since $n_1>8$, it follows from Lemma \ref{2.2} and \ref{2.4} that $b=c=1$ and consequently $P_n=2p_1p_2$ which implies that $v_2(d)\geq 3$. Hence, the only possible value of $d$ is $8$ and from Equation \eqref{3.1} it follows that 

\begin{align*}
8\cdot\frac{10^m-1}{9}&=\phi(P_n)=\phi(2p_1p_2)\\&=(p_1-1)(p_2-1)\\
&=p_1p_2+1-(p_1+p_2)\\
&=\frac{P_n}{2}+1-(p_1+p_2).
\end{align*} 
Therefore,
$$p_1+p_2=\frac{P_n}{2}+1-8\cdot\frac{10^m-1}{9}$$ and $p_1p_2=\frac{P_n}{2}$. Thus, the quadratic equation with $p_1$ and $p_2$ as roots is :
$$x^2-\bigg(\frac{P_n}{2}+1-8\cdot \frac{10^m-1}{9} \bigg)x+\frac{P_n}{2}=0$$
Since this equation has integer solutions, it's discriminant $\bigtriangleup$ must be a perfect square. But
\begin{align*}
\bigtriangleup &= \bigg(\frac{P_n}{2}+1-8\cdot \frac{10^m-1}{9} \bigg)^2-4\cdot \frac{P_n}{2}\\
               &\equiv (6P_n+1-8\cdot9^{-1}((-1)^m-1) )^2-2P_n ~(\bmod~11) .           
\end{align*}
If $m$ is even, then 
\begin{align*}
\bigtriangleup &\equiv (6P_n+1)^2-2P_n\equiv 36P_n^2+1+10P_n\equiv 3P_n^2-P_n+1~(\bmod~11).
\end{align*}
In view of Table 1, the possible values of $P_n$ such that $3P_n^2-P_n+1$ is a quadratic residue modulo $11$ are $P_n\equiv 0,4,6,9(\bmod~11)$.
\begin{enumerate}
\item If $P_n\equiv 0(\bmod~11)$ then $n\equiv 0,12(\bmod~24)$ and consequently $n_1\equiv 0,6(\bmod~12)$. Hence, $p_1=P_{n_1}\equiv0,10(\bmod~20)$ which is not possible since $p_1$ is a prime.\\

 \item If $P_n\equiv4(\bmod~11)$ then $n\equiv6,7,17(\bmod~24)$. Since 
$n=2n_1$, the only possibility for $n_1$ is $n_1\equiv3(\bmod~12)$. But this implies that $p_1= P_{n_1}\equiv5(\bmod~20)$ which implies that $P_{n_1}=5,$ which implies that $n_1=3$ and consequently $n=6$ which is not possible since we have assumed that $n>16$.\\
 \item If $P_n\equiv 6(\bmod~11)$ then $n\equiv 9,15(\bmod~24)$. Since $n$ is even, this is not possible.\\
 \item If $P_n\equiv9(\bmod~11)$ then $n\equiv14,22(\bmod~24)$ and then $n_1\equiv 7,11(\bmod~12)$. Consequently, $p_1= P_{n_1}\equiv1,9(\bmod~20)$ and $p_2\equiv Q_{n_1}\equiv19(\bmod~20)$. Thus, $$\phi(P_n)=(p_1-1)(p_2-1)\equiv0,4(\bmod~20).$$ Substitution in Equation \eqref{3.1} gives  $18\equiv0,4(\bmod~20)$, which is not possible.
\end{enumerate}
If $m$ is odd then
 \begin{align*}
  \bigtriangleup&\equiv (6P_n-8\cdot9^{-1}((-1)^m)-1 )^2-2P_n \\ 
                &\equiv (6P_n+4)^2-2P_n\\
                &\equiv 3P_n^2+2P_n+5~(\bmod~11)
 \end{align*}
Modulo 11, the possible values of $P_n$ such that $3P_n^2+2P_n+5$ is a quadratic residue, are $6$ and $7$. We have seen that Equation \eqref{3.1} has no solution when $P_n\equiv 6(\bmod~11)$. If $P_n\equiv7(\bmod~11)$ then $n\equiv 5,18(\bmod~24)$. Since $n=2n_1$ is even, the only possible value of $n$ is 18. Thus, $n_1\equiv 9(\bmod~12)$. Consequently, $p_1= P_n\equiv5(\bmod~20)$ which implies that $P_{n_1}=5$ and hence $n_1=3$. But, this is not possible since we have assumed that $n_1>8$.\\
 \par
Secondly, assume that there exist only one odd prime factor $p_1$ of $P_n$. If $n$ is even (say $n=2n_1$), then $P_n$ is even and then $P_n=2P_{n_1}Q_{n_1}=2^ap_1^b$. Then, of course, $P_{n_1}=2^{a-1}$ and $Q_{n_1}=p_1^{b}$. 
If $P_{n_1}=2^{a-1}$, then in view of Lemmas \ref{2.2} and \ref{2.4} the possible values of $n_1$ are $n_1=1,2$ and consequently $n=2,4$ which contradicts our assumption of $n>16$. If $n$ is odd, then $P_n=p_{1}^b$. Once again by virtue of Lemma \ref{2.2} if $P_n=p_{1}^b$ and $b\geq 2$, then the only non-trivial possibility is $n=7$ which contradicts our assumption that $n>16$. Hence $b=1$ and consequently $P_n=p$. Therefore $\phi(P_n)=p-1=P_n-1$ which is a multiple of 4. Thus, $d\in\{4,8\}$.\\

If $d=4$, then
$$P_n=4\cdot \frac{10^m-1}{9}+1=\frac{4\cdot 10^m+5}{9}$$
which is a multiple of 5. Hence, $P_n$ is not a prime. If $d=8$, then Equation (\ref{3.1}) can be written as 
\begin{align}\label{3.3}
9P_n-1=8\cdot 10^m=2^{m+3}5^m.
\end{align}
Since $m\geq1$, from Table 1 $9P_n-1\equiv0(\bmod~40)$ and then $n\equiv7,17(\bmod~24)$. But the Pell sequence modulo 11 has period 24. If $n\equiv7,17(\bmod~24)$, then $P_n\equiv4(\bmod~11)$. Reducing Equation (\ref{3.3}) modulo 11, we get $35\equiv8\cdot10^m(\bmod~11)$ which finally results in $3\equiv\pm1(\bmod~11),$ which is not possible. Hence, $\phi(P_n)$ cannot be a repdigit consisting of atleast two digits for any natural number $n$.
\end{proof}
\section{Repdigits in Euler functions of associated Pell numbers}

In this section, we prove that if the  Euler function of an associated Pell number $Q_n$ is of the form $d(10^m-1)/9$ for some $n,m\in N$ and $d \in \{1,2,\cdots,9\}$, then it satisfies some conditions which is discussed in several theorems. We need the least residues and periods of associated Pell sequence $\{Q_n\}_{n\geq0}$ modulo $4$, $5$, $8$, $20$ to prove the main results of this section. We list them in the following table.
\begin{table}[h!]
  \centering
 \caption{Periods of $Q_n$}
  \label{tab:table2}
  \begin{tabular}{|l|c|r|}
  \hline
    $k$ & $Q_n(\bmod~k)$ & Period\\
    
    \hline
    4 & 1, 1, 3, 3 & 4\\
    \hline
    5 &  1, 1, 3, 2, 2, 1, 4, 4, 2, 3, 3, 4 & 12\\ 
    \hline
    8 & 1, 1, 3, 7 & 4\\
    \hline
    20 & 1, 1, 3, 7, 17, 1, 19, 19, 17, 13, 3, 19 & 12\\
    \hline
    
  \end{tabular}
\end{table}
\par
For $n\leq16$, it is easy to see 1, 3, 7 are the only numbers in the associated Pell sequence $\{Q_n\}_{n\geq 0}$ such that  $\phi(Q_n)$ is a repdigit. If $\phi(Q_n)$ is a repdigit for some $n>16$, then by Lemma \ref{2.10}, $\phi(Q_n)\geq \frac{2}{\sqrt{3}}\sqrt{Q_n}\geq \frac{2}{\sqrt{3}}\sqrt{Q_{17}}>450390>10^5$ and thus $m>5$. Further, since $Q_n$ is odd for all integers $n$, its canonical decomposition can be written as
\begin{align}\label{4.1}
Q_n=p_1^{a_1}p_2^{a_2}\cdots p_r^{a_r},
\end{align}
where $r\geq0$, $p_1,p_2,\ldots,p_r$ are distinct odd primes and $a_i>0$ for all $i=1,2,\ldots,r$.
Then 
\begin{align}\label{4.2}
\phi(Q_n)=p_1^{a_1-1}p_2^{a_2-1}\cdots p_r^{a_r-1}(p_1-1)(p_2-1)\cdots(p_r-1).
\end{align}
If $\phi(Q_n)=d(10^m-1)/9$ for some $n$ then
\begin{align}\label{4.3}
v_2(\phi(Q_n))=\sum_{i=1}^rv_2(p_i-1)=v_2(d),
\end{align}
where $v_2(d)\in\{0,1,2,3\}$. 

The following theorem deals with the nonexistence of repdigits for $d \neq4,8$.


\begin{theorem}\label{4.1}
If $d\in\{1,2,3,5,6,7,9\}$ then there is no positive integer solutions of the equation 
\begin{equation}\label{4.4}
\phi{(Q_n)}=d\cdot\frac{10^m-1}{9}
\end{equation}
for any $n>16$.
\end{theorem}
\begin{proof}
In view of Equation (\ref{4.3}), if $v_2(d)=0$, then $d\in\{1,3,5,7,9\}$ and consequently $\phi(Q_n)$ is odd, which is only possible for $n=0,1$. If $v_2(d)=1$, then $d=2,6$ and in view of Equation \eqref{4.3}, $r=1$ and $p_1=3(\bmod~4)$. In particular, $Q_n=p_1^{a_1}$ and an application of Lemma \ref{2.4} gives $a_1=1$ and hence
\begin{align*}
\phi(Q_n)=Q_n-1=d\cdot \frac{10^m-1}{9},
\end{align*}
 implies that
\begin{align*}
Q_n=d\cdot \frac{10^m-1}{9}+1.
\end{align*}
If $d=2$, then $Q_n=3(\bmod~20)$ and using Table 2, we get $n\equiv2,10(\bmod~12)$. Now, it follows from Lemma \ref{2.9} that $n$ is either prime or a power of $2$. When $n\equiv2(\bmod~12)$, we can write $n$ as $n=2(6k+1)$ for some $k\geq0$, which is a power of 2 only if $k=0$. Hence $n=2$ which contradicts our assumption that $n>16$. If $n\equiv10(\bmod~12)$, then we can write $n$ as  $n=2(6l+5)$ for some $l\geq0$ which is not a power of 2 for any $l$. If $d=6$, then $Q_n=7(\bmod~20)$ and  from Table 2 it follows that $n\equiv3(\bmod~12)$. Thus, $n=3(4k+1)$ which is a prime only if $k=0$. Hence the only possibility left is $n=3$, which contradicts our assumption that $n>16$.
\end{proof}

The following theorem gives a structure of prime factors of $Q_n$ when $\phi(Q_n)$ is a repdigit consisting of $4$'s or $8$'s.
\begin{theorem}\label{4.2}
If $n$ is even, $\phi(Q_n)$ is a repdigit for some $n>16$, i.e, $\phi(Q_n)=d\cdot \frac{10^m-1}{9}$ where $d\in\{4,8\}$, then all prime factors in the canonical decomposition of $Q_n$ are congruent to $3$ modulo $4$.
\end{theorem}
\begin{proof}
Since $n$ is even, we can write $n$ as $n=2^t \cdot n_1$ where $t,n_1\in \mathbb{N}$ and $n_1$ is odd. Assume to the contrary that there exists a prime $p\equiv1(\bmod~4)$ in the prime factorization of $Q_n$. It follows from Equation (4.3) that number of prime factors of $Q_n$ is at most 2, that is $r\leq 2$. If $d=4$ then $r=1$ and $Q_n=p_1^{a_1}$. By virtue of Lemma \ref{2.4}, it follows that $Q_{n}=p_1$. Hence,
  \begin{align*}
  \phi(Q_{n})=Q_{n}-1=4\cdot \frac{10^m-1}{9}\equiv 4(\bmod~10).
  \end{align*}
Thus,
  \begin{align*}
  Q_{n}\equiv 5(\bmod~10),
  \end{align*}
which implies that $5|Q_{n}$ which is not true in view of Lemma \ref{2.1}. Reducing Identity $(2)$ of Lemma \ref{2.1}, modulo $p_1$ where $p_1|Q_n$ and $p_1\equiv1(\bmod~4)$, we get $-2P_n^2\equiv1(\bmod~p)$ which implies that $\big(\frac{-1}{p}\big)=1$ and hence $\big(\frac{2}{p}\big)=1$. Thus, $p\equiv1(\bmod~8)$ and hence $d=8$, $r=1$ and $Q_{n}=p_1^{a_1}$ which implies that $Q_{n}=p_1$ by Lemma \ref{2.4}. Since $Q_{2^t}|Q_n$, there exist primitive prime factors of $Q_{2^t}$ and $Q_n$ which divides $Q_n$. Hence, $n=2^t$, otherwise $Q_n$ have more than one prime factors which contradicts $r=1$. Now,
   
   \begin{align*}
  \phi(Q_{2^t})=Q_{2^t}-1=8\cdot \frac{10^m-1}{9}\equiv 3(\bmod~5)
  \end{align*}
 implies that  $Q_{2^t}\equiv 4(\bmod~5)$ which is not possible since $Q_{2^t}\equiv2, 3(\bmod~5)$ (see Table \ref{tab:table2}). Hence there does not exist any prime factor $p$ such that $p\equiv1(\bmod~4)$.

\end{proof}

The following theorem deals with the nonexistence of repdigits when $n$ is even.
\begin{theorem}\label{4.3}
If $\phi(Q_n)$ is a repdigit consisting of $4$'s or $8$'s, then $n$ is odd. 
\end{theorem}
\begin{proof}
Assume that $Q_n$ has a canonical decomposition $Q_n=p_1^{a_1}\cdots p_r^{a_r}$ and $\phi(Q_n)$ is a repdigit consisting of $4$'s and $8$'s. We will show that $n$ is odd. Assume to the contrary that $n$ is even, say $n=2^t\cdot n_1$, where $t\geq1$ and $n_1$ is odd. In view of Theorem \ref{4.2}, if $p_i$ is a prime factor of $Q_n$ then $p_i\equiv 3(\bmod~4)$ and consequently $r=2$ for $d=4$ and $r=3$ for $d=8$, where $r$ is the number of distinct prime factors of $Q_n$. Reducing  Identity (2) of Lemma \ref{2.1} modulo $p_i$, we get $\big(\frac{2}{p_i}\big)=-1$. Since $p_i$ is a primitive factor of $Q_{2^t\cdot l}$ for some divisor $l$ of $n_1$, by virtue of Lemma \ref{2.7}, it follows that $p_i\equiv-1(\bmod~2^t\cdot l)$ and thus $p_i\equiv-1(\bmod~2^t)$ for $i=1,2,\ldots, r$.\\
 
  If $d=4$ then 
 \begin{align*}
 d\cdot\frac{10^m-1}{9}=\phi(Q_n)&=\phi(p_1^{a_1}p_2^{a_2})\\
 &=p_1^{a_1-1}p_2^{a_2-1}(p_1-1)(p_1-1)\\
 &\equiv \pm4(\bmod~2^t).
 \end{align*}
 Similarly, if $d=8$, then $d\cdot\frac{10^m-1}{9}\equiv \pm8(\bmod~2^t)$ which implies that $10^m\equiv 10,-8(\bmod~2^{max\{0,t-r\}})$ for $r\in\{2,3\}$. If $10^m\equiv 10(\bmod~2^{max\{0,t-r\}})$ then  $t\leq3$ and if $10^m\equiv -8(\bmod~2^{max\{0,t-r\}})$ then $t\leq5$. Hence, we conclude that  $t\leq5$. Since $Q_{2^t}|Q_{n}$, it follows that
\begin{align}\label{4.4}
Q_{n}=Q_{2^t}p_1^{a_1}
\end{align}
or
 \begin{align}\label{4.5}
Q_{n}=Q_{2^t}p_2^{a_2}p_3^{a_3}.
\end{align}
We distinguish two cases.
\par
\textbf{Case 1:} $n_1$ is not a multiple of 3.\\
\par
If $n_1$ is not a multiple of 3 then $Q_{2^t \cdot {n_1}}\equiv Q_{2^t}(\bmod~5)$ for $t\leq5$ and consequently Equation (\ref{4.4}) reduces to $Q_{2^t}\equiv Q_{2^t}p_1^{a_1}(\bmod ~5)$ and hence, $p_1^{a_1}\equiv1(\bmod~5)$. If $a_1$ is odd then $p_1\equiv1(\bmod ~5)$ and $5|p_1-1|\phi(Q_n)=d\cdot \frac{10^m-1}{9}$, which is not possible since $d\in \{4,8\}$. If $a_1$ is even then $\frac{Q_n}{Q_{2^t}}=p_1^{a_1}$ which is impossible in view of Lemma \ref{2.5}. On the other hand, if Equation (\ref{4.5}) holds then for $t\in\{2,3,4\}$, $Q_{2^t}$ is a prime and $Q_{2^t}\equiv1(\bmod~16)$ and then $16|\phi(Q_{2^t})|\phi(Q_n)=8\cdot \frac{10^m-1}{9}$. This implies that $2|\frac{10^m-1}{9}$, which is a contradiction. Now if $t=5$ then $Q_{2^5}=257 \cdot 1409 \cdot 2448769$. Since all the prime factors of $Q_{2^5}$ are congruent to 1 modulo 16, it follows that $16^3|\phi(Q_{2^t})|8\cdot \frac{10^m-1}{9}$, which is not possible.\\
 
 Finally, if $t=1$ then $Q_{2^t \cdot n_1}\equiv 3(\bmod~4)$. Reducing Equation (\ref{4.5}) modulo 4, we get $3\equiv 3^{a_2+a_3+1}(\bmod~4)$, and therefore $a_2+a_3$ must be even. If $a_2$ is even, then $a_3$ is also even and then $\frac{Q_{n}}{3}$ is a perfect square, which is not possible in view of Lemma \ref{2.5}. Hence $a_2$ and $a_3$ are odd. Since $n_1$ is odd and $(n_1,3)=1$, it follows that $2n_1\equiv 2(\bmod~4)$  and $2n_1\equiv2,4(\bmod~6)$ implying that $2n_1\equiv 2,10(\bmod~12)$. Since the period of $\{Q_n\}_{n\geq0}$ modulo 8 is 12, it follows that $Q_{2n_1}\equiv 3(\bmod~8)$ if $2n_1\equiv 2,10(\bmod~12)$. Reducing Equation (\ref{4.5}) modulo 8, we get $p_2^{a_2}p_3^{a_3}\equiv 1(\bmod~8)$. Since $a_2$ and $a_3$ are odd, $p_2p_3\equiv1(\bmod~8)$, which, together with $p_2\equiv p_3\equiv 3(\bmod~4)$ implies that $p_2\equiv p_3 \equiv 3,7(\bmod~8)$ and then $\frac{p_2-1}{2}\equiv \frac{p_3-1}{2}\equiv \pm1(\bmod~4)$. Thus, $\frac{p_2-1}{2}\cdot \frac{p_3-1}{2}\equiv1 (\bmod~4)$.
Thus, 
\begin{align*}
\phi(Q_n)&=\phi(3p_2^{a_2}p_3^{a_3})=2 p_2^{a_2-1}p_3^{a_3-1}(p_2-1)(p_3-1)\nonumber \\ 
         &=8\cdot \frac{p_2-1}{2}\cdot \frac{p_3-1}{2}\cdot p_2^{a_2-1}p_3^{a_3-1}=8\cdot\frac{10^m-1}{9}
\end{align*}  
implies that 
\begin{align}\label{4.6}
\frac{p_2-1}{2}\cdot \frac{p_3-1}{2}\cdot p_2^{a_2-1}p_3^{a_3-1}=\frac{10^m-1}{9}.
\end{align}
Since $$\frac{p_2-1}{2}\cdot \frac{p_3-1}{2}\cdot p_2^{a_2-1}p_3^{a_3-1}\equiv 1(\bmod~4)$$ and $\frac{10^m-1}{9}\equiv 3 (\bmod~4)$, it follows from Equation \eqref{4.6} that $1\equiv 3(\bmod~4)$, which is a contradiction.\\

\textbf{Case 2:}  $n_1$ is a multiple of 3.\\

If $n$ is even and $3|n_1$, then $n$ is of the form $n=2^t\cdot3^s\cdot n_2$ where $t,s\geq1$ and $n_2$ is odd. If $n_2>1$, then by Lemma \ref{2.8}, $Q_n$ is a multiple of the primitive factors of $Q_{2^t}$, $Q_{3 \cdot2^t}$, $Q_{2^t \cdot n_2}$ and $Q_{2^t\cdot 3n_2}$ which implies that $r\geq4$, which is not true since $r\leq 3$. Hence $n_2=1$ and $n=2^t\cdot3^s$. If $s\geq2$, then $Q_n$  is divisible by its own primitive prime factor as well as primitive prime factors of $Q_{2^t}$, $Q_{3\cdot 2^t}$ and $Q_{3^2\cdot 2^t}$, again this is not true since $r\leq3$. Hence $n$ is either $n=3^2\cdot 2^t$ or $n=3\cdot 2^t $. If $n=3\cdot 2^t$, then $Q_{3\cdot 2^t}=Q_{2^t}(4Q_{2^t}^2-3)$. If $t \geq2$ then $(Q_{2^t},4Q_{2^t}^2-3)=(Q_{2^t},3)=1$ and if $t=1$ then $(Q_{2^t},4Q_{2^t}^2-3)=3$. Assume that $t\geq2$. Then $Q_{2^t}$ and $4Q_{2^t}^2-3$ are relatively prime and hence we can write
\begin{align}\label{4.7}
Q_{2^t}=p_3^{a_3},
\end{align}
\begin{align}\label{4.8}
4Q_{2^t}^2-3=p_4^{a_4},
\end{align}
where both $p_3,p_4\equiv3(\bmod~4)$ are primes and in view of Lemma \ref{2.4}, $a_3=1$ and $Q_{2^t}=p_3$. Reducing Equation (\ref{4.8}) modulo 4, we get $3^{a_4}\equiv 1(\bmod~4)$ which implies that $a_4$ is even and consequently  $4Q_{2^t}^2-3$ is a perfect square, which is possible only when $Q_{2^t}=1$ i.e. $t=0$ which contradicts $t\geq2$. Hence we are left with $t=1$ and then $n=6$, which contradicts our assumption that $n>16$.\\

If $n=2^t\cdot3^2$ then $Q_{2^t3^2}=Q_{2^t}(4Q_{2^t}^2-3)(4Q_{2^t3}^2-3)$. If $t\geq2$, then the factors on the right hand side are relatively prime and can be written as
\begin{align}\label{4.9}
Q_{2^t}=p_5^{a_5},
\end{align}
 \begin{align}\label{4.10}
4Q_{2^t}^2-3=p_6^{a_6},
\end{align}
\begin{align}\label{4.11}
4Q_{2^t3}^2-3=p_7^{a_7},
\end{align}
where $p_5,p_6,p_7\equiv3(\bmod~4)$ are all primes. Since no terms of the associated Pell sequence is a perfect power by Lemma \ref{2.4}, it follows that $a_5=1$. Hence Equation \eqref{4.10} implies that $a_6$ is even. Therefore $4Q_{2^t}^2-3$ is a perfect square. As we have seen it has no solution for any $t\geq2$. Hence, we are left with case $t=1$, that is $n=18$. One can easily check that $\phi(Q_{18})$ is not a repdigit. Hence $n$ is odd.
\end{proof}

\begin{theorem}\label{4.4}
If $\phi(Q_n)=4\cdot \frac{10^m-1}{9}$, then $n$ is an odd prime such that $n^2|10^m-1$.
\end{theorem}
\begin{proof}
 If $\phi(Q_n)=4\cdot \frac{10^m-1}{9}$ then $n$ is odd in view of Theorem \ref{4.3} and if $p$ is a prime factor of $Q_n$ then by Theorem \ref{4.2} $p\equiv3(\bmod~4)$, which implies that $r=2$ and $p\equiv3,7(\bmod~8)$. Reducing $Q_n^2-2P_n^2=-1$ modulo $p$, we get $2P_n^2\equiv 1(\bmod~p)$ which implies that $\big(\frac{2}{p}\big)=1$ and hence $p\equiv7(\bmod~8)$ and $\frac{p-1}{2}\equiv3(\bmod~4)$. If $Q_n$ has the prime factorization $Q_n=p_1^{a_1}p_2^{a_2}$, then
\begin{align*}\label{4.12}
\phi(Q_n)=4p_1^{a_1-1}p_2^{a_2-1}\cdot \frac{p_1-1}{2}\cdot\frac{p_2-1}{2}=4\cdot\frac{10^m-1}{9}
\end{align*}
 implies that $3^{a_1+a_2-2}\equiv 3(\bmod~4)$. But it holds only when $a_1+a_2-2$ is odd and so that $a_1$ and $a_2$ are of opposite parity. If $n=pq$ then $Q_{n}$ is divisible by the primitive prime factors of $Q_p$, $Q_q$ and $Q_{n}$, which is not true since $r=2$. Hence $n=p^2$ or $p$. If $n=p^2$ then $Q_{p^2}=Q_p\cdot \frac{Q_{p^2}}{Q_p}$ where both the factors are relatively prime and therefore we take $Q_p=p_1^{a_1}$ and $\frac{Q_{p^2}}{Q_p}=p_2^{a_2}$. Since one of $a_1$ and $a_2$ is even, either $Q_p$ or $\frac{Q_{p^2}}{Q_p}$ is a square, which is not possible in view of Lemma \ref{2.4}, \ref{2.5}. If $n=p$, all prime factors of $Q_p$ are primitive prime factors of $Q_p$ and $p_1,p_2\equiv 1(\bmod~p)$, which implies that $p^2|(p_1-1)(p_2-1)|\phi(Q_p)|4\cdot\frac{10^m-1}{9}$ and hence $p^2|10^m-1$. 
 \end{proof}

\textbf{Remark:} Using the proof of Theorem \ref{3.1}, one can also conclude thet the Euler function of non of the odd indexed balancing number $B_n$ is a repdigit with atleast two digits, since $P_{2n}=2B_n$( See \cite{Behera 1999, Panda 2011,Ray 2009} ). One can also conclude that no Euler function of Lucas balancing number $C_n$ is a repdigit since Lucas balancing numbers are nothing but even indexed associated Pell numbers( See \cite{Behera 1999,Ray 2009} ) and $\phi(Q_{2n})$ cannot be a repdigit in view of Theorem \ref{4.3}. It would be interesting  to find a bound of $n$ for which repdigits in Euler functions of associated Pell numbers exist and explore more properties corresponding to odd prime $n$ in Theorem \ref{4.4}. We leave these as open problems to the reader. 

\medskip

\begin{thebibliography}{99}
\bibitem{Behera 1999}
A. Behera and G. K. Panda, \emph{On the square roots of triangular numbers}, Fib. Quart., 37(2) (1999), 98 - 105.\\

\bibitem{Bravo 2015} J. J. Bravo, P. Das, S. G. Sánchez and S. Laishram, \emph{Powers in products of terms of Pell's and Pell–Lucas Sequences,}  International J. Number Theory 11(04) (2015), 1259--1274.\\

 \bibitem{Bravo 2016} J. J. Bravo, B. Faye, F. Luca, A. Tall, \emph{Repdigits as Euler functions of Lucas numbers,} An. St. Univ. Ovidius, Constanta, Vol.24(2) (2016), 105–-126.\\
 
\bibitem{carm 1913} R. D. Carmichael, \emph{On the Numerical Factors of Arithmetic Forms $\alpha^n -\beta^n$. } Ann. of Math. 15 (1913-1914):30--70.\\
\bibitem{Cohn 1996} J. H. E. Cohn, \emph{Perfect Pell powers,} Glasgow Math. J. 38 (1996), 19–-20.\\

\bibitem{Damir 2014} M. T. Damir, B. Faye, F. Luca and A. Tall, \emph{Members of Lucas sequences whose Euler function is a power of 2}, Fib. Quart. (2014).\\

\bibitem{Faye 2015} B. Faye, F. Luca , \emph{Pell and Pell-Lucas numbers with only one distinct digit}, Ann. Math. et Informaticae. 45 (2015), 55--60.\\
\bibitem{koshy} T. Koshy, \emph{Pell and Pell–Lucas numbers with applications,} Springer,(2014).\\

\bibitem{lehmer 1932} D.H. Lehmer, \emph{On Euler totient function},~Bull. Amer. Math. Soc.  \textbf{38} (1932), 745--751.\\

\bibitem{Luca 1997} F. Luca. \emph{Euler indicators of Lucas sequences}, Bull. Math. Soc. Math. Roumanic Tome 40(88) (1997), 3--4.\\

\bibitem{Luca 1998} F. Luca. \emph{On the equation $\phi(x^m-y^m)=x^n+y^n$}, Bull. Irish Math. Soc.40 (1998) 46–-55.\\

\bibitem{Luca 1998} F. Luca,  \emph{Arithmetic functions of Fibonacci numbers,}  (1997-1998).\\

\bibitem{Luca 1999}F. Luca, On the equation $\phi(|x^m+y^m|)=|x^n+y^n|$, Indian J. Pure Appl. Math. 30 (1999), 183–-197.\\

\bibitem{Luca 2002} F. Luca,  \emph{Euler indicators of binary recurrence sequences,}  Collect. Math. 53, 2 (2002), 133–-156.\\

\bibitem{Luca 2006} F. Luca and M. Mignotte, \emph{$\phi(F_{11})=88$,}  Divulgaciones Mat. 14 (2006), 101–-106.\\

\bibitem{Luca 2014} F. Luca and P. Stanica, \emph{Equations with arithmetic functions of Pell numbers,}  Bull. Math. Soc. Sci. Math. Roumanie Tome 57(105) No. 4 (2014), 409–-413.\\
 
 \bibitem{Daniel 1998} W. L. McDaniel and P. Ribenboim, \emph{Square-classes in Lucas sequences having odd parameters,}  J. Number Theory 73 (1998), 14–-27.\\
 
 \bibitem{Daniel 2000} W. L. McDaniel, \emph{On Fibonacci and Pell numbers of the form $kx^2$, } Fibonacci quarterly,(2000).\\
 \bibitem{Panda 2011}
G. K. Panda and P.K. Ray \emph{Some links of balancing and
cobalancing numbers with Pell and associated Pell numbers}, 
Bull. Inst. Math., Acad. Sinica (New Series), \textbf{6(1)} (2011),
41-72.\\
 
 \bibitem{Petho 1992} A. Peth˝o, \emph{The Pell sequence contains only trivial perfect powers.} Sets, graphs and numbers (Budapest, 1991), 561–-568, Colloq. Math. Soc. J´anos Bolyai, 60, North–Holland, Amsterdam, 1992. MR94e:11031.\\
 \bibitem{Ray 2009} P. K. Ray, \emph{Balancing and cobalancing numbers, Ph.D. Thesis}, National Institute of Technology, Rourkela, India, (2009).\\

 \bibitem{pell} \url{https://en.wikipedia.org/wiki/Pell_number}.\\
 
\end{thebibliography}
 \end{document}